\documentclass[12pt]{amsart}

\usepackage{latexsym, amssymb, amsmath,ulem,soul,esint}
\usepackage{tikz}

\usepackage{amsfonts, graphicx}
\usepackage{graphicx,color}
\newcommand{\be}{\begin{equation}}
\newcommand{\ee}{\end{equation}}
\newcommand{\beq}{\begin{eqnarray}}
\newcommand{\eeq}{\end{eqnarray}}

\usepackage{wasysym,stmaryrd}
\theoremstyle{definition}
\newtheorem{ex}{Example}

\newtheorem{thm}{Theorem}
\newtheorem{cl}{Corollary}

\newtheorem{prop}{Proposition}
\newtheorem{lm}{Lemma}
\newtheorem{clm}{Claim}

\newtheorem{rmk}{Remark}

\usepackage{mathrsfs} 

\theoremstyle{remark}
\numberwithin{equation}{section}

\def\be{\begin{equation}}
\def\ee{\end{equation}}

\newcommand{\R}{\mathbb{R}}

\newcommand{\M}{\mathcal{M}^{n}}

\DeclareMathOperator{\Ric}{Ric}

\DeclareMathOperator{\vol}{Vol}

\DeclareMathOperator{\inj}{inj}
\DeclareMathOperator{\Rm}{Rm}

\DeclareMathOperator{\scal}{scal}

\begin{document}

\title[]
{Rotationally symmetric Ricci Flow on $\mathbb{R}^{n+1}$}

\author{Ming Hsiao}
\address[Ming Hsiao]{Department of Mathematics
National Taiwan University
and
National Center for Theoretical Sciences,
Math Division,
Taipei 10617,
Taiwan
}
\email{r12221002@ntu.edu.tw}

\renewcommand{\subjclassname}{
  \textup{2020} Mathematics Subject Classification}

\date{\today}

\begin{abstract}
We establish a short-time existence theory for complete Ricci flows under scaling-invariant curvature bounds, starting from rotationally symmetric metrics on $\mathbb{R}^{n+1}$ that are noncollapsed at infinity, without assuming bounded curvature. As a consequence, we construct a complete Ricci flow solution coming out of a rotationally symmetric metric, which has a cone-like singularity at the origin and no minimal hypersphere centered at the origin, using an approximation method.
\end{abstract}


\maketitle

\markboth{Ming Hsiao}{}

\section{Introduction}
 
Given a smooth manifold $\M$, a \textit{Ricci flow} is a family of smooth Riemannian metrics $(\M,g(t)))_{t\in I}$ satisfying
\begin{equation}
\frac{\partial}{\partial t}g(t)=-2\Ric(g(t)),
\end{equation}
on $\M\times I$, where $I\subset\R$ is some interval. This was first introduced by R. Hamilton \cite{10.4310/jdg/1214436922} in 1983 and proved that a closed 3-manifold with positive Ricci curvature is diffeomorphic to a spherical space form by evolving the Ricci flow from the given metric, which is an elegant geometric result, and to date, at least for the author, no proof is known that avoids the Ricci flow approach. Following Hamilton's work, geometers developed many deep connections between geometry and topology via the Ricci flow. The most famous and celebrated work is the Poincaré conjecture and the geometrization conjecture, which G. Perelman proved in 2002 through his groundbreaking works on the Ricci flow \cite{perelman2002entropy,perelman2003finiteextinctiontimesolutions,perelman2003ricciflowsurgerythreemanifolds}. 
\medskip

 On the closed manifold, the existence and uniqueness of a Ricci flow solution with a given initial metric were first established by R. Hamilton \cite{10.4310/jdg/1214436922} using the "Nash-Moser iteration method". Later, D. Deturck \cite{Deturck} provided another approach using the "fixed gauge method", which simplified Hamilton's original proof. As a result, these fundamental properties are well understood in the compact setting. However, those properties only progress partially in the non-compact setup. W.-X. Shi established the breakthrough existence results in the complete non-compact setting \cite{10.4310/jdg/1214443292} in 1989, under an additional assumption that the initial metric obeys the \textbf{bounded curvature} assumption. In the past decade, some results have removed the bounded curvature assumption; instead, they assumed some curvature lower bounds \cite{Yi2019almostnonnegative,lee2024threemanifoldsnonnegativelypinchedricci}. A remarkable achievement in 2D is the construction of \textit{instantaneously complete Ricci flow} starting from rough initial data. G. Giesen and P. Topping were the first to establish this result for smooth metrics on Riemann surfaces \cite{GiesenTopping2011}. Later, P. Topping and H. Yin extended it to the case of non-atomic Radon measures on Riemann surfaces \cite{ToppingYin2021}.
 \medskip
 
In this paper, we study the existence of Ricci flow starting from the rotationally symmetric initial data on $\R^{n+1}$. If we assume the flow is unique, which is not always clear but has been verified under scaling invariant estimate \cite{lee2025uniquenessricciflowscaling}, the Ricci flow $g(t):=\sigma(x,t)^{2}dx^{2}+f(x,t)^{2}g_{\text{std}}$ can be written as partial differential equations involving two scalar functions, the warped function $f$ and the change of distance function $\sigma$, where $g_{\text{std}}$ is the standard spherical space form on $S^{n}$. This reduces the complexity of equations, and it seems easier to analyze. However, without assuming bounded curvature, the existence of complete Ricci flow still lacks development. In K\"{a}hler geometry, the existence of K\"{a}hler-Ricci flow with an $U(n)$-invariant K\"{a}hler metric on $\mathbb{C}^{n}$ is analogous to what we are considering. We remark that some results removing the bounded curvature condition, possibly under slightly weaker assumptions. We refer interested readers to \cite{CLT17, YZ13} for further details.
\medskip
        
Our main result demonstrates the existence of \textbf{rotationally symmetric Ricci flow (RSRF)} starting from a rotationally symmetric metric whose \textbf{warped function has a uniform lower bound at infinity}. Intuitively, the manifold bounds a thin neck at infinity, accordingly, the small neck should be able to act as a barrier, bridling the original metric. If we allow the warped function to degenerate, then the situation could become a puzzle. For instance, in \cite{CG16}, the authors constructed a complete and rotationally symmetric Riemannian manifold $(\R^{3},g)$ such that there is \textbf{no complete and $\kappa$-noncollapsed Ricci flow with initial data $(\R^{3},g)$}. Their construction involves infinitely many discrete minimal hyperspheres, centered at the origin $o$ and with the diameter that tends to zero as the distance goes to infinity. These observations guide us to connect the relation between the warped function and the volume ratio of the rotationally symmetric manifold, making these moral viewpoints serious. More rigorously, we prove the following. 
\begin{thm}\label{thm-of-nondegenerated-warped-function}
    Let $(\R^{n+1},g=ds^{2}+f(s)^{2}g_{\text{std}})$ be a complete and rotational symmetric manifold with $\liminf_{s\rightarrow\infty}f(s)>0$. Then there exists a complete and RSRF $(\R^{n+1},g(t))_{t\in[0,T_{g}]}$ with $g(0)=g$ and a constant $\Lambda_{g}>0$ such that
    \begin{equation}
        |{\Rm(g(t))}|\leq\frac{\Lambda_{g}}{t},
    \end{equation}
    on $\R^{n+1}\times(0,T_{g}]$. Furthermore, assuming there are some constants $\varepsilon,\delta,\ell>0$ such that
        \begin{equation}
        \begin{cases}
             |{\Rm(g)}|\leq \ell^{-2}, \text{ on }B_{g}(o,6\ell);\\
            f_{s}\geq\delta, \text{ on }B_{g}(o,3\ell);\\
            f\geq\varepsilon \ell, \text{ on }\R^{n+1}\setminus B_{g}(o,\frac{1}{2}\ell).
        \end{cases}
    \end{equation}
        Then $\Lambda_{g}=\Lambda(n,\delta,\varepsilon)>0$ and $T_{g}= T(n,\delta,\varepsilon)\ell^{2}>0$. 
\end{thm}
It is worth noting that our result \textbf{does not require any global curvature conditions}. It can be seen in Section \ref{section1.5} that, even if we only assume the almost PIC1 condition on the rotationally symmetric metric, the warped function can only grow exponentially; the curvature lower bound constraints the growth of the warped function subtly. However, our result has no limitation on the growth order of the warped function. Consequently, this result provides a novel example of Ricci flow in higher dimensions (and does not come from the product with a 2D solution constructed in  \cite{GiesenTopping2011,ToppingYin2021}) originating from spaces with unbounded curvature and with bounded curvature immediately. 
\medskip
    
The proof is based on a pseudolocality theorem for RSRF (Lemma \ref{prioriestimate}) and an a priori estimate for the volume ratio along RSRF (Lemma \ref{noncollapsing for bounded curvature v2}). One drawback of Theorem \ref{thm-of-nondegenerated-warped-function} is the dependency of $T_{g}$ and $\Lambda_{g}$, \textit{the two-sided curvature bound near the origin}. It's unreal to assume the two-sided curvature bound if we aim to approximate the rotationally symmetric metric with a singularity at the origin. In practice, with a uniform two-sided bound, the limit should be at least a $\mathcal{C}^{1,\alpha}$-metric. Nevertheless, if we consider the case of \textbf{non-decreasing warped function}, then we can relax the local curvature condition near the $o$. 

\begin{thm}\label{theoremofnoncollapsing}
        Let $(\R^{n+1},g:=ds^{2}+f(s)^{2}g_{\text{std}})$ be a complete and rotationally symmetric manifold. Suppose that $f_{s}(s)\geq0$ on $\R^{n+1}$. Then there exist two constants 
        $\Lambda_{g}>0$ and $T_{g}>0$ and a complete and  RSRF $(\R^{n+1},g(t))_{t\in[0,T_{g}]}$ starting from $g$ such that \begin{equation*}
            |{\Rm}(g(t))|\leq\frac{\Lambda_{g}}{t}
        \end{equation*}
        on $\R^{n+1}\times(0,T_{g}]$. Furthermore, suppose that $\ell,\delta>0$ be two constants so that $\scal(g)\geq-\ell^{-2}$ and $f_{s}\geq\delta$ on $B_{g}(o,4\ell)$. Then $\Lambda_{g}=\Lambda(n,\delta)>0$ and $T_{g}=T(n,\delta)\ell^{2}>0$.
    \end{thm}

The flows we constructed in both theorems are indeed the same, due to the uniqueness theorem \cite[Theorem 1.1.]{lee2025uniquenessricciflowscaling}. The main difference is to improve the estimates of $\Lambda_{g}$ and $T_{g}$. In addition, it can be seen that after assuming nonnegative scalar curvature and a positive lower bound of $f_{s}(\cdot,0)$, the surviving time can be extended to $+\infty$.
\begin{cl}\label{long-time solution}
        Let $(\R^{n+1},g:=ds^{2}+f(s)^{2}g_{\text{std}})$ be a complete and rotationally symmetric manifold. Suppose that $f_{s}(s)\geq \delta>0$ on $\R^{n+1}$ for some constant $\delta$ and $\scal(g)\geq0$ on $\R^{n+1}$. Then there exists a constant 
        $\Lambda(n,\delta)>0$ and a complete and   RSRF $(\R^{n+1},g(t))_{t\geq0}$ starting from $g$ such that 
        \begin{equation}\label{eq-of-cl1}
            |{\Rm}(g(t))|\leq\frac{\Lambda}{t}
        \end{equation}
        on $\R^{n+1}\times(0,\infty)$. 
\end{cl}   
    One application of Theorem \ref{theoremofnoncollapsing} is to smooth out a specific metric on $\R^{n+1}$, which is regular outside the origin with some additional conditions, by Ricci flow smoothing.
    \begin{cl}\label{roughinitialdata}
        Let $(\R^{n+1},g:=ds^{2}+f(s)^{2}g_{\text{std}})$ be an unbounded rotationally symmetric manifold outside $o$ with \textit{uniform Ricci lower bound}. Suppose that $f_{s}(s)\geq0$ on $\R^{n+1}$ and at $s=0$, $f_{s}(0)\in(0,1)$ and $f^{(k)}(0)=0$ hold for all $k\geq2$. Then there exist two constants 
        $\hat{\Lambda}(n,g)>0$ and $\hat{T}(n,g)>0$ and a complete RSRF $(\R^{n+1},g(t))_{t\in(0,\hat{T}]}$ such that \begin{equation*}
            |{\Rm}(g(t))|\leq\frac{\hat{\Lambda}}{t}
        \end{equation*}
        on $\R^{n+1}\times(0,\hat{T}]$, and $g(t)$ converges to $g$ smoothly outside $o$ as $t\rightarrow0$. Moreover, $(\R^{n+1},d_{g(t)},o)$ converges to $(\R^{n+1},d_{g},o)$ in the pointed Gromov-Hausdorff sense. Furthermore, if $\scal(g)\geq0$ and $f_{s}(s)\geq \delta>0$ outside $o$, then $\hat{T}=+\infty$ and $\hat{\Lambda}=\hat{\Lambda}(n,\delta)$.
    \end{cl}
    We adopt an approximation approach on the initial metric $g$ and quote the equivariant compactness theorem of Ricci flow \cite[Proposition 3.7.]{BHZ24}. The uniform Ricci lower bound assumption plays a linchpin in joining the initial data smoothly by \cite[Theorem 1.6.]{DSS22}. 
    \medskip
    
    This paper is organized as follows: In Section \ref{section1.5}, we discuss an elementary property of the complete and rotationally symmetric metric on $\R^{n+1}$ and establish a lower bound for volume ratio by the warped function. In Section \ref{section1.6}, we prove a pseudolocality result based on the scaling invariant estimate on the Weyl tensor and investigate two a priori estimates for volume ratio. Also, we present the proofs of Theorem \ref{thm-of-nondegenerated-warped-function} and \ref{theoremofnoncollapsing} and Corollary \ref{long-time solution} in Section \ref{proof of ain theorem}. Finally, we perform the approximation approach of Corollary \ref{roughinitialdata} rigorously in Section \ref{approxmateion section}.
\vskip0.2cm

{\it Acknowledgement}: The author would like to thank his advisor, Yng-Ing Lee, for bringing this question to his attention and for her valuable discussions and comments. Also, the author is grateful to Pak-Yeung Chan and Man-Chun Lee for their interest and for pointing out related works in the K\"{a}hler geometry. The author was supported by the National Science and Technology Council under grant No. 112-2115-M-002-004-MY3 and a direct grant from the National Center for Theoretical Sciences.

\section{An elementary property of Rotationally symmetric metric on $\R^{n+1}$}\label{section1.5}
A rotationally symmetric metric on $\R^{n+1}$ is a Riemannian metric $g=ds^{2}+f(s)^{2}g_{\text{std}}$ with a smooth warped function $f:\R^{+}\rightarrow\R^{+}$ satisfying
\begin{equation}\label{metriccondition}      \lim_{s\rightarrow0}f^{(1)}(s)=1\text{ and }\lim_{s\rightarrow0}f^{(2k)}(s)=0,
\end{equation}
for all $k\geq0$, where $g_{\text{std}}$ is the standard spherical metric on $S^{n}$. According to \cite[Corollary 1.2.]{Yi2019almostnonnegative}, one of the existence of Ricci flows starting from $g$ is given by the conditions $\Rm(g)+\alpha_{0}\mathcal{I}\in\mathcal{C}_{\text{PIC1}}$ and volume is weakly non-collapsed. The following proposition shows whether $g$ satisfies the weak volume non-collapsed condition. 
\begin{prop}[Volume Ratio bounds]\label{volumenoncollapseforrota}
    Let $g=ds^{2}+f(s)^{2}g_{\text{std}}$ be a complete and rotationally symmetric metric on $\R^{n+1}$. Suppose there exists a constant $\delta>0$ such that 
    \begin{equation}
        f(s)\geq\delta s\chi_{[0,1]}(s)+\delta\chi_{(1,\infty)}(s),
    \end{equation} 
    on $\R^{n+1}$. Then there exists a constant $v=v(\delta,n)>0$ such that $\frac{\vol_{g}(B_{g}(x,r))}{r^{n+1}}\geq v$ for all $x\in\R^{n+1}$ and $r\in(0,1]$. In particular, if $f_{s}\geq0$ on $\R^{n+1}$ and $f_{s}\geq\delta$ on $|s|\leq1$, then the assumption holds.
\end{prop}
\begin{proof}
    Due to the rescaling property, it suffices to deal with the case $r=1$. For $x=o$, the co-area formula implies
    \begin{equation}
        \vol_{g}(B_{g}(o,\frac{1}{4}))=\int_{0}^{\frac{1}{4}}f(s)^{n}\omega_{n}ds\geq\delta^{n}\int_{0}^{\frac{1}{4}}s^{n}\omega_{n}ds=\frac{\left(\frac{1}{4}\right)^{n+1}\delta^{n}\omega_{n}}{n+1}=:v_{1},
    \end{equation}
    where $\omega_n$ is the volume of unit sphere $S^n$. Note that $B_{g}(o,\frac{1}{4})\subset B_{g}(x,1)$ for all $x\in B_{g}(o,\frac{1}{2})$, this shows $\vol_{g}(B_{g}(x,1))\geq v_{1}$ for all $x\in B_{g}(o,\frac{1}{2})$. For $x\notin B_{g}(o,\frac{1}{4})$, $f(x)\geq\frac{\delta}{4}$. It remains to the case $x\notin B_{g}(x,\frac{1}{2})$. Note that 
    \begin{equation}
       A^{\frac{1}{4}}_{x}\subset B_{g}(x,\frac{1}{2}),
    \end{equation}
    where
    \begin{equation}
         A_{x}^{r}:=\left\{(s',\theta')\in\R^{+}\times S^{n}:|s'-s|<r\text{ and }d_{g_{\text{std}}}(\theta,\theta')<\frac{r}{f(s')}\right\},
    \end{equation}
    for all $r\geq0$. By the co-area formula again,  
    \begin{align*}
        \vol_{g}(A^{\frac{1}{4}}_{x})=&\int_{s-\frac{1}{4}}^{s+\frac{1}{4}}f(s')^{n}\vol_{g_{\text{std}}}\left(B_{g_{\text{std}}}\left(\theta,\frac{1}{4f(s')}\right)\right)ds'\\
        =&\int_{s-\frac{1}{4}}^{s+\frac{1}{4}}f(s')^{n}\int_{0}^{\min\left\{\frac{1}{4f(s')},\pi\right\}}{\omega_{n-1}\sin^{n-1}(t)dt}ds'.
    \end{align*}
    If $f(s')<1/\pi$, then
\begin{equation}\label{eq42}
    \int_{0}^{\min\left\{\frac{1}{4f(s')},\pi\right\}}{\omega_{n-1}\sin^{n-1}(t)dt}\geq\int_{0}^{\frac{\pi}{4}}{\omega_{n-1}\sin^{n-1}(t)dt}=:A_{n}.
\end{equation}
Otherwise, 
\begin{equation}\label{eq44}
\begin{aligned}
\int_{0}^{\min\left\{\frac{1}{4f(s')},\pi\right\}}{\omega_{n-1}\sin^{n-1}(t)dt}\geq&~ \omega_{n-1}\int_{0}^{\frac{1}{4f(s')}}\left(t\cos\frac{\pi}{4}\right)^{n-1}dt\\
=&~\frac{\omega_{n-1}}{n}2^{-\frac{n-1}{2}}\left(\frac{1}{4f(s')}\right)^{n}.  
\end{aligned}
\end{equation}
Finally, (\ref{eq42}) and (\ref{eq44}) imply
\begin{equation}
    \begin{aligned}
        \vol_{g}(A_{x}^{\frac{1}{4}})=&\int_{s-\frac{1}{4}}^{s+\frac{1}{4}}f(s')^{n}\int_{0}^{\min\left\{\frac{1}{4f(s')},\pi\right\}}{\omega_{n-1}\sin^{n-1}(t)dt}ds'\\
        \geq&\int_{s-\frac{1}{4}}^{s+\frac{1}{4}}f(s')^{n}\min\left\{A_{n},\frac{\omega_{n-1}}{n}2^{-\frac{n-1}{2}}\left(\frac{1}{4f(s')}\right)^{n}\right\}ds'\\
        \geq&\int^{s+\frac{1}{4}}_{s-\frac{1}{4}}\min\left\{A_{n}\left(\frac{\delta}{4}\right)^{n},B_{n}\right\}ds'\\
        =&\frac{1}{2}\min\left\{A_{n}\left(\frac{\delta}{4}\right)^{n},B_{n}\right\}=:v_{2}.
    \end{aligned}
\end{equation}
Taking $v:=\min\{v_{1},v_{2}\}$, only depends on $\delta,n$, we have shown that $\vol_{g}(B_{g}(x,1))\geq v_{0}$ for all $x\in\R^{n}$, which confirms the assertion.
\end{proof}
We now have a sufficient condition for weakly volume non-collapsed, so we lack whether we have  $\Rm(g)+\alpha_{0}\mathcal{I}\in\mathcal{C}_{\text{PIC1}}$. We adopt the same notations as in \cite[Section 2]{SD04} to derive the curvature of metric $g$. Let $K$ and $L$ be the sectional curvatures of the $2$-plane perpendicular to and tangent to the sphere $\{x\}\times S^{n}$ respectively. Then
\begin{equation}\label{sectionalcurvofRSmetric}
    K=-\frac{f''(s)}{f(s)}\ \text{ and }\ L=\frac{1-f'(s)^{2}}{f(s)^{2}}.
\end{equation}
Also, the Ricci curvature and scalar curvature are
\begin{equation}\label{eqofRic}
    \Ric(g)=\frac{-nf''(s)}{f(s)}ds^{2}+\left[-f(s)f''(s)+(n-1)(1-f'(s)^{2})\right]g_{\text{{std}}},
\end{equation}
and 
\begin{equation}
    \scal(g)=n\left(-2\frac{f''(s)}{f(s)}+(n-1)\frac{1-f'(s)^{2}}{f(s)^{2}}\right).
\end{equation}
Note that $(\R^{n+1},g)$ is locally conformally flat, so that the curvature operator can be written as
\begin{equation}
    \Rm(g)=W(g)+\frac{1}{n-1}\left(\Ric(g)-\frac{\scal(g)}{2n}g\right)\owedge g=\frac{1}{n-1}\left(\Ric(g)-\frac{\scal(g)}{2n}g\right)\owedge g,
\end{equation}
where $\owedge$ is the Kulkarni–Nomizu product. Since $\mathcal{I}=g\owedge g$, we get
\begin{equation}
    \Rm(g)+\mathcal{I}=\frac{1}{n-1}\left\{\left[\Ric(g)-\left(\frac{\scal(g)}{2n}-(n-1)\right)g\right]\owedge g\right\}=:A_{g}\owedge g.
\end{equation}
By the definition of the Kulkarni–Nomizu product, $\Rm(g)+\mathcal{I}\in\mathcal{C}_{\text{PIC1}}$ is equivalent to 
\begin{equation}
    A_{g}(e_{1},e_{1})+A_{g}(e_{2},e_{2})+2A_{g}(e_{3},e_{3})\geq0,
\end{equation}
for all orthonormal three $g$-frames $\{e_{1},e_{2},e_{3}\}$. By a direct computation, we get
\begin{equation}
    A_{g}=\left(1-\frac{f''(s)}{f(s)}-\frac{1-f'(s)^{2}}{2f(s)^{2}}\right)ds^{2}+\left(\frac{1-f'(s)^{2}}{2}+f(s)^{2}\right)g_{\text{std}}.
\end{equation}
So one can easily check that the following condition is equivalent to   $\Rm(g)+\mathcal{I}\in\mathcal{C}_{\text{PIC1}}$.
\begin{equation}\label{PIC1condition}
    \begin{cases}
        \frac{1-f'(s)^{2}}{2f(s)^{2}}+1\geq0,\text{ }\forall s>0;\\
        2-\frac{f''(s)}{f(s)}\geq0,\text{ }\forall s>0.        
    \end{cases}
\end{equation}
We focus on the first condition. It implies $f'(s)\leq\sqrt{1+2f(s)^{2}}$ for all $s\geq0$. This forces  $f(s)\leq\left(\exp\left(\sqrt{2}s\right)-1\right)/\sqrt{2}$ and $f_{s}(s)\leq\exp\left(\sqrt{2}s\right)$ for all $s\geq0$. Therefore, the lower curvature bounds restrict the warped function's growth. 
\medskip

In summary, as a special case of \cite[Corollary 1.2.]{Yi2019almostnonnegative}, the following result holds:
\begin{ex}
    Let $f:[0,\infty)\rightarrow[0,\infty)$ be a smooth function satisfying (\ref{metriccondition}) and (\ref{PIC1condition}). Then for any $n\geq4$, there is a complete and RSRF $(\mathbb{R}^{n+1},g(t))_{t\in[0,T]}$ with initial metric $g(0)=ds^{2}+f(s)^{2}g_{\text{std}}$.
\end{ex}
The above discussion expounds that we shouldn't require the curvature to be bounded from below at the initial time if we allow the growth of the warped function $f(s,0)$ faster than the exponential growth. On the other hand, under the assumption of Theorem \ref{thm-of-nondegenerated-warped-function}, one could make the growth of $f$ as large as we want.

\section{Volume ratio and Pseudolocality}\label{section1.6}
\subsection{Pseudolocality theorem}
The first key observation is the pseudolocality theorem for RSRF, which follows from the non-collapsing property. The validity of this theorem is supported by the generalized Hamilton-Ivey estimate in this context.
\begin{lm}[Pseudolocality theorem]\label{prioriestimate}
Suppose that  $(\mathcal{M}^{n},g(t))_{t\in[0, T]}$ is a Ricci flow so that for some $x_{0}\in\mathcal{M}^{n}$, 
\begin{enumerate}
    \item $B_{t}(x_{0},2)\Subset\mathcal{M}^{n}$ for $t\in[0,T]$;
    \item $|\text{W}(g(x,t))|\leq\frac{c}{t}$ for all $x\in B_{t}(x_{0},2)$ and $t\in[0,T]$, for some $c\geq0$;
    \item\label{volumeratiolowerbound} $\frac{\vol_{t}(x,r)}{r^{n}}\geq v_{1}>0$ for all $t\in[0,T]$ and $x\in B_{t}(x_{0},1)$ and $r\in(0,1)$.
\end{enumerate}
Then there exists two constants $a=a(v_{1},c,n)>0$ and $\hat{T}=\hat{T}(v_{1},c,n)$ so that for all $t\in(0,\min\{T,\hat{T}\}]$ we have
\begin{equation}
    |\text{Rm}(g(x_{0},t))|\leq\frac{a}{t}.
\end{equation}
\end{lm}
\begin{proof}[Proof of Lemma \ref{prioriestimate}]
   We argue by contradiction. Suppose the conclusion is false. Then, there exists $v_{0}>0$ and $c\geq0$ such that for any $a_{i}\rightarrow\infty$ and $T_{i}\rightarrow0$, there exists a sequence of Ricci flow $(\mathcal{M}_{i}^{n},g_{i}(t))_{t\in[0, T_{i}]}$ and $x_{i}\in\mathcal{M}_{i}^{n}$ satisfying the assumptions, but the assertion fails in arbitrary small time. By the smoothness of Ricci flow, we may find $t_{i}\in(0,T_{i}]$ so that
    \begin{enumerate}
    \item $B_{g_{i}(t)}(x_{i},2)\Subset\mathcal{M}_{i}^{n}$ for $t\in[0,t_{i}]$;
    \item $|{\text{W}(g_{i}(x,t))}|\leq \frac{c}{t}$ for all $t\in[0,t_{i}]$ and $x\in B_{g_{i}(t)}(x_{i},2)$;
    \item $\frac{\vol_{t}(x,r)}{r^{n}}\geq v_{1}>0$ for all $t\in[0,t_{i}]$, $x\in B_{t}(x_{i},1)$ and $r\in(0,1]$;
    \item\label{curvaturedecay00} $|{\Rm(g_{i}(x_{i},t))}|< a_{i}/t$ for all $t\in(0,t_{i})$;
    \item $|{\Rm(g_{i}(x_{i},t_{i}))}|=a_{i}/t_{i}$.
\end{enumerate}
 We may assume $a_{i}t_{i}\rightarrow0$. By (\ref{curvaturedecay00}) and $a_{i}t_{i}\rightarrow0$, \cite[Lemma 5.1.]{simon2017local} implies that for $i$ large enough, there is a constant $\beta=\beta(n)>0$, $\tilde{t}_{i}\in(0,t_{i}]$, and $\tilde{x}_{i}\in B_{g_{i}(\tilde{t}_{i})}(x_{i},\frac{3}{4}-\frac{1}{2}\beta\sqrt{a_{i}\tilde{t}_{i}})$ so that
\begin{equation}
    |\text{Rm}(g_{i}(x,t))|\leq 4|\text{Rm}(g_{i}(\tilde{x}_{i},\tilde{t}_{i}))|=:4Q_{i},
\end{equation}
whenever $d_{g_{i}(\tilde{t}_{i})}(x,\tilde{x}_{i})<\frac{1}{8}\beta a_{i}Q_{i}^{-\frac{1}{2}}$ and $\tilde{t}_{i}-\frac{1}{8}a_{i}Q_{i}^{-1}\leq t\leq\tilde{t}_{i}$ where $\tilde{t}_{i}Q_{i}\geq a_{i}\rightarrow\infty$. Now, consider the parabolic scaling centered at $(\tilde{x}_{i},\tilde{t}_{i})$, that is, define $\tilde{g}_{i}(t):=Q_{i}g_{i}(Q_{i}^{-1}t+\tilde{t}_{i})$ for all $t\in[-\frac{1}{8}a_{i},0]$. In the proof of \cite[Lemma 5.1.]{simon2017local}, the parabolic domain $B_{\tilde{g}_{i}(0)}(\tilde{x}_{i},\frac{1}{8}\beta a_{i})\times[-\frac{1}{8}a_{i},0]$ is contained in the region that assumption holds. Then we get
\begin{enumerate}
    \item $|{\Rm(\tilde{g}_{i}(\tilde{x}_{i},0))}|=1$;
    \item $|\text{Rm}(\tilde{g}_{i}(x,t))|\leq4$ for all  $(x,t)\in B_{\tilde{g}_{i}(0)}(\tilde{x}_{i},\frac{1}{8}\beta a_{i})\times[-\frac{1}{8}a_{i},0]$;
    \item\label{volumecondition} $\frac{\vol_{\tilde{g}_{i}(t)}(x,r)}{r^{n}}\geq v_{1}>0$ for all $(x,t)\in B_{\tilde{g}_{i}(0)}(\tilde{x}_{i},\frac{1}{8}\beta a_{i})\times[-\frac{1}{8}a_{i},0]$ and $r\in(0,\sqrt{Q_{i}}]$;
    \item\label{scalingstatement}$|\text{W}(\tilde{g}_{i}(x,t))|\leq c(t+Q_{i}\tilde{t}_{i})^{-1}$ for all  $(x,t)\in B_{\tilde{g}_{i}(0)}(\tilde{x}_{i},\frac{1}{8}\beta a_{i})\times[-\frac{1}{8}a_{i},0]$.
\end{enumerate}
Now, by Hamilton's compactness theorem and Cheeger-Gromov-Taylor's classical result, after passing to a subsequence, there is a geometric limit 
 \begin{equation}
 (\mathcal{M}^{n}_{i},\tilde{g}_{i}(t),(\tilde{x}_{i},0))_{t\in[-\frac{1}{8}a_{i},0]}\rightarrow(\mathcal{N}^{n},g_{\infty}(t),(x_{\infty},0))_{t\leq0},    
 \end{equation}
 converge in Hamilton-Cheeger-Gromov's sense. Also, $(\mathcal{N}^{n},g_{\infty}(t))_{t\leq0}$ is a non-flat, with bounded curvature (less than $4$), LCF on each time-slice, $\text{AVR}(g_{\infty}(t))\geq v_{1}$ for all $t\leq0$, ancient solution of complete Ricci flow. By the generalized Hamilton-Ivey estimate \cite[Theorem 1.1.]{ZHANG2015774}, $(\mathcal{N}^{n},g_{\infty}(t))$ has nonnegative curvature operator for all $t\leq0$. But this violates that $\text{AVR}=0$ for $\kappa$-solution  \cite[Proposition 11.4]{perelman2002entropy}.
\end{proof}
\subsection{A priori estimates for volume ratio}
Now, we aim to prove that the condition     (\ref{volumeratiolowerbound}) in Lemma \ref{prioriestimate} is preserved under the curvature decayed $|{\Rm(g(t))}|\leq c/t$. We refer to the following lemma about the lower bound of scalar curvature being preserved under the curvature decayed assumption.
\begin{lm}\label{localscalarcurvaturebound}\cite[Lemma 8.1.]{simon2017local}
For any constants $c,K>0$, $\gamma\in(0,1)$ and $n\in\mathbb{N}$, there exists a constant $\hat{T}(c,K,\gamma,n)>0$ satisfying the following. Let $(\M,g(t))_{t\in[0,T]}$ be a Ricci flow. Suppose that $B_{g(t)}(x_{0},1)\Subset\M$ for some $x_{0}\in\M$ for all $t\in[0,T]$ and
\begin{enumerate}
    \item $\scal(g(0))\geq-K$ on $B_{g(0)}(x_{0},1)$.
    \item $\Ric(g(t))\leq\frac{c}{t}$ on $B_{g(t)}(x_{0},\sqrt{t})$ for all $t\in(0,T]$.
\end{enumerate}
Then 
\begin{equation}
    \scal(g(t))\geq-2K,
\end{equation}
on $B_{g(0)}(x_{0},1-\gamma)$ for all $t\in[0,\min\{T,\hat{T}\}]$.
\end{lm}
Recall the shrinking ball lemma by Simons-Topping \cite[Corollary 3.3]{simon2017local}, which is based on Perelman-Hamilton's distance distortion lemma \cite[Lemma 8.3.]{perelman2002entropy}.
\begin{lm}\cite[Corollary 3.3]{simon2017local}\label{distancelemma}
    There is a constant $\beta=\beta(n)$ satisfying the following. For any 
    Ricci flow $(\M,g(t))_{t\in[0,T]}$, with the estimate $\Ric(x,t)\leq\frac{(n-1)c}{t}$ on $B_{0}(x_{0},r)\times(0,T]$, where $B_{0}(x_{0},r)\Subset\M$. Then for all $0\leq s\leq t<T$,
    \begin{equation}\label{shrinkingballeq}
        B_{t}(x_{0},L-\beta\sqrt{ct})\subset B_{s}(x_{0},L-\beta\sqrt{cs}),
    \end{equation}
    and 
    \begin{equation}
        d_{t}(x_{0},y)-\beta\sqrt{ct}\geq d_{s}(x_{0},y)-\beta\sqrt{cs},
    \end{equation}
    for all $y\in B_{t}(x_{0},L-\beta\sqrt{ct})$.
\end{lm}
We also need the following local maximal principle due to Tam and Lee.
\begin{lm}\cite[Theorem 1.1.]{Lee_Tam_2022}\label{localmaximalpriciplev2}
    Let $(\M,g(t))_{t\in[0,T]}$ be a Ricci flow which is possibly incomplete. Suppose that $\Ric(g(t))\leq \alpha t^{-1}$ on $\M\times(0,T]$ for some $\alpha>0$. Let $\varphi(x,t)$ be a continuous function on $\M\times[0,T]$ which satisfies $\varphi(x,t)\leq\alpha t^{-1}$ on $\M\times(0,T]$ and 
\begin{equation}
    \left.\left(\frac{\partial}{\partial t}-\Delta_{g(t)}\right)\varphi\right|_{(x_{0},t_{0})}\leq L(x_{0},t_{0})\varphi(x_{0},t_{0}),
\end{equation}
whenever $\varphi(x_{0},t_{0})>0$ in the sense of barrier, for some continuous function $L(x,t)$ on $\M\times[0,T]$ with $L(x,t)\leq\alpha t^{-1}$. Suppose $p\in\M$ and $R>0$ such that $B_{0}(p,2)\Subset\M$ and $\varphi(x,0)\leq0$ on $B_{0}(p,2)$. Then for any $\ell>\alpha+1$, there exists $\hat{T}(n,\alpha,\ell)>0$ such that for $t\in(0,\min\{T,\hat{T}\}]$,
\begin{equation}\label{estimateofvarphi}
    \varphi(p,t)\leq t^{\ell}.
\end{equation}
\end{lm}
Now, we declare an a priori estimate of the volume ratio under the $c/t$ estimate in the case of $f_{s}\geq0$.
    \begin{lm}[Noncollapsing Property with $f_{s}\geq0$]\label{noncollapsingforboundedcurvature}
    Let $(\mathbb{R}^{n+1} ,g(t):=ds_{t}^{2}+f(s_{t},t)^{2}g_{\text{std}})_{t\in[0,T]}$ be a complete RSRF with \textbf{bounded curvature}. Then $f_{s}(s,0)\geq0\Rightarrow f_{s}\geq0$ on $\mathbb{R}^{n+1}\times[0,T]$. Furthermore, assume  $|{\text{Rm}}(g(t))|\leq \alpha/t$ on $\mathbb{R}^{n+1}\times[0,T]$, and there exists two constants $\ell,\delta>0$ such that at $t=0$,
    \be
    \begin{cases}
        \scal(g(\cdot,0))\geq-\ell^{-2},\text{ on }B_{0}(o,4\ell);\\
        f_{s}(\cdot,0)\geq\delta, \text{ whenever }|s_{0}|\leq4\ell.
    \end{cases}
    \ee
    Then $\exists v(n,\delta),\hat{T}(n,\alpha,\delta)>0$ such that 
    \begin{equation}
        \frac{\text{Vol}_{g(t)}B_{g(t)}(x,r)}{r^{n+1}}\geq v,
    \end{equation}
    on $\mathbb{R}^{n+1}\times[0,\min\{T,\hat{T}\ell^{2}\}]$ and for all $r\in(0,\ell]$.
\end{lm}
\begin{rmk}
    It's crucial that $v$ is independent of $\alpha$ in the proof of Theorem \ref{theoremofnoncollapsing}.
\end{rmk}
\begin{proof}[Proof of Lemma \ref{noncollapsingforboundedcurvature}]
    The first assertion follows from \cite[Lemma 3.1.]{DIGIOVANNI2021107621}. It suffices to confirm the last assertion. After performing the parabolic rescaling, we may assume $\ell=1$ and $\hat{T}\leq T$. By Lemma \ref{localscalarcurvaturebound}, $\scal(g(t))\geq-2$ on $B_{0}(o,3)\times[0,T_{1}]$ for some $T_{1}=T_{1}(n,\alpha)>0$. Note that the evolution equation of $f_{s}$ is
    \begin{equation}\label{equationoff_s}
        \left(\frac{\partial}{\partial t}-\Delta_{g(t)}\right)f_{s}=\frac{\scal(g(t))}{n}f_{s},
    \end{equation}
    we refer it to \cite[Lemma 3.1.]{DIGIOVANNI2021107621}. Define $\phi(x,t):=\delta(1-2t)-f_{s_{t}}(x,t)$, then (\ref{equationoff_s}) implies 
    \begin{equation}
    \begin{aligned}
        \left(\frac{\partial}{\partial t}-\Delta_{g(t)}\right)\phi=&~\frac{\scal(g(t))}{n}\phi-\frac{\delta(1-2t)\scal(g(t))}{n}-2\delta \\
        \leq&~\frac{\scal(g(t))}{n}\phi+{2\delta }\left(\frac{1-2t}{n}-1\right)\\
        \leq&~\frac{\scal(g(t))}{n}
\phi,        
    \end{aligned}
    \end{equation}
    on $B_{0}(o,3)\times[0,\min\{T_{1},\frac{1}{4}\}]$. Note that $\phi<\frac{\delta}{2}$ and $\phi(0)\leq0$ holds on $B_{0}(o,3)\times[0,\min\{T_{1},\frac{1}{4}\}]$, by Theorem  \ref{localmaximalpriciplev2}, we obtain $\phi\leq\frac{\delta}{4}$ on $B_{0}(o,2)\times[0,\min\{T_{1},\frac{1}{4},T_{2}\}]$ for some constant $T_{2}(n,\alpha,\delta)>0$. Therefore, we get $f_{s}\geq\frac{\delta}{4}$ on $B_{0}(o,2)\times[0,\hat{T}]$ for some $\hat{T}(n,\alpha,\delta)>0$. By Lemma \ref{distancelemma}, after shrink $\hat{T}$ again, $s_{t}(x)\geq s_{0}(x)-1$ holds on $B_{t}(o,1)$ for all $t\leq\hat{T}$. That is, $f_{s_{t}}\geq\frac{\delta}{4}$ on $B_{t}(o,1)$ for all $t\in[0,\hat{T}]$. Combining this with $f_{s}\geq0$, we get $f(s_{t},t)\geq\frac{\delta}{4}s_{t}\chi_{[0,1]}(s_{t})+\frac{\delta}{4}\chi_{(1,\infty)}(s_{t})$ on $\R^{n+1}\times[0,\hat{T}]$. Thus, we complete the proof by applying Proposition \ref{volumenoncollapseforrota} on $\R^{n+1}\times[0,\hat{T}]$.

\end{proof}

The next lemma concerns an a priori estimate of the volume ratio when the initial metric bounds a cylinder outside a compact set. It seems stronger than Lemma \ref{noncollapsingforboundedcurvature}, but requires the two-sided bounds of curvature at $t=0$. 

\begin{lm}[Noncollapsing Property \textbf{without} assuming $f_{s}\geq0$]\label{noncollapsing for bounded curvature v2}
    Let $(\mathbb{R}^{n+1} ,g(t):=ds_{t}^{2}+f(s_{t},t)^{2}g_{\text{std}})_{t\in[0,T]}$ be a complete RSRF with $f(s,0)\geq v\ell$  holds on $s\geq M\ell$ for some constants $v,M,\ell>0$. Furthermore, assume  $|{\text{Rm}}(g(t))|\leq \alpha/t$ on $\mathbb{R}^{n+1}\times[0,T]$ and at $t=0$, there exists some constants $w,\delta,\varepsilon>0$ such that
    \begin{equation}
        \begin{cases}
            |{\Rm(g(0))}|\leq \ell^{-2}, \text{ on }B_{0}(o,6\ell);\\
            |{\Rm(g(0))}|\leq (w\ell)^{-2}, \text{ outside }B_{0}(o,M\ell);\\
            f_{s}(s,0)\geq\delta, \text{ on }B_{0}(o,3\ell);\\
            f\geq\varepsilon \ell, \text{ on }B_{0}(o,(M+4w)\ell)\setminus B_{0}(o,\frac{1}{2}\ell).
        \end{cases}
    \end{equation}
    Then $\exists v_{1}(n,v,\delta,\varepsilon),\hat{T}(n,v,\alpha,\delta,\varepsilon,w)>0$ such that 
    \begin{equation}
        \frac{\text{Vol}_{g(t)}B_{g(t)}(x,r)}{r^{n+1}}\geq v_{1},
    \end{equation}
    on $\mathbb{R}^{n+1}\times[0,\min\{T,\hat{T}\ell^{2}\}]$ and for all $r\in(0,\ell]$.
\end{lm}
\begin{rmk}
The main challenge here is the absence of the lower bound of $f_{s}$ since the lower bound $f_{s}(\cdot,t)\geq-\alpha/t$ is necessary when applying the Lee-Tam local maximal principle. In the proof of Lemma \ref{noncollapsingforboundedcurvature}, $f_{s}\geq0$ follows from the maximal principle. But the same discussion fails in this scenario. We adopt another approach, although it depends on some stronger global assumptions and the \textbf{local two-sided curvature bounds}.   
\end{rmk}
\begin{proof}[Proof of Lemma \ref{noncollapsing for bounded curvature v2}]
  We first show that near the origin, the slope of $f$ has a uniform lower bound in terms of initial data. By the scaling-invariant property, we may assume $\ell=1$ and $T\geq\hat{T}$. From now on, $\hat{T}$ denotes a positive constant which only depends on $n,v,\alpha,\delta,\varepsilon,w$ and may change from line to line. By \cite[Corollary 3.2.]{10.4310/jdg/1246888488}, $|{\Rm}|\leq e^{c_{1}(n)\alpha}=:K'$ on $B_{0}(o,5)\times[0,\hat{T}]$ for some constant $c_{1}(n)>0$. Therefore, we adopt the calculation as at the end of Section \ref{section1.5}. By Lemma \ref{distancelemma},  $|f_{s_{t}}(\cdot,t)|\leq\sqrt{1+K'f^{2}(\cdot,t)}$ on $B_{t}(o,4)\times\{t\}$ for all $t\leq\hat{T}$, which implies
 \begin{equation}
     \left|\log(\sqrt{K'}f(s_{t},t)+\sqrt{1+K'f(s_{t},t)^{2}})\right|\leq\sqrt{K'}s_{t}\Rightarrow f(s_{t},t)\leq\frac{e^{\sqrt{K'}s_{t}}}{\sqrt{K'}},
 \end{equation}
 for all $s_{t}\leq 4$. Fixed such $t\leq\hat{T}$. Let $\tilde{s}:=\inf\{s\geq0: f_{s_{t}}|_{(s_{t},t)=(s,t)}\leq-1-e^{4\sqrt{K'}}\}>0$ be a positive constant. If $\tilde{s}<4$, then
 \begin{equation}
    \frac{e^{8\sqrt{K'}}}{f(\tilde{s},t)^{2}}\leq\frac{f_{s_{t}}(\tilde{s},t)^{2}-1}{f(\tilde{s},t)^{2}}\leq K',
 \end{equation}
follows from the uniform curvature bound. Combining those two estimates, we receive $\tilde{s}\geq 4>0$ and $-f_{s_{t}}\leq1+e^{4\sqrt{K'}}$ on $B_{t}(o,4)\times\{t\}$ for all $t\leq \hat{T}$. Due to the uniform curvature bound and Lemma \ref{distancelemma}, 
\begin{equation}\label{eq-of-set-inclusion}
    B_{0}(p,e^{-2(n-1)K't}r)\subset B_{t}(p,r)\subset B_{0}(p,r+\beta\sqrt{\alpha t}),
\end{equation}
as long as $B_{0}(p,r+\beta\sqrt{\alpha t})\subset B_{0}(o,5)$ and $t\leq\hat{T}$, where $\beta(n)>0$. Take $e^{-2(n-1)K'\hat{T}}\geq1/2$ and $\beta\sqrt{\alpha\hat{T}}\leq1/4$. Similar to the proof of Lemma \ref{noncollapsingforboundedcurvature}, the function $\phi(s_{t},t):=\delta(1-K't)-f_{s_{t}}(s_{t},t)$ satisfies  $\phi(x,t)\leq\alpha/t$ on $B_{0}(o,3)\times(0,\hat{T}]$ (possibly shrink $\hat{T}$ again). The same discussion as in the proof of Lemma \ref{noncollapsingforboundedcurvature} shows $f(s_{t},t)\geq\frac{\delta}{4}s_{t}$ on $B_{t}(o,1)$ for all $t\in[0,\hat{T}]$. Combining this and (\ref{eq-of-set-inclusion}) (where we apply $s_{t}\geq s_{0}-\beta\sqrt{\alpha t}$), $f(x,t)\geq\frac{\delta}{16}$ on $s_{0}^{-1}(1/2)\times[0,\hat{T}]$.
    
Now, notice that for $s_{0}>M$, the curvature norm of $g(0)$ is less than $w^{-2}$. By \cite[Corollary 3.2.]{10.4310/jdg/1246888488}, $|{\Rm(g(t))}|\leq e^{c_{1}\alpha}w^{-2}=K'w^{-2}$ on $[0,\hat{T}]$ whenever $s_{0}>M+2w$. By the evolution equation of $f$ and the expression of sectional curvature \cite[Section 2.1.]{DIGIOVANNI2021107621}, 
\begin{equation}\label{eq-of-f}
    f_{t}=f_{s_{t}s_{t}}-(n-1)\frac{1-f_{s_{t}}^{2}}{f}\Rightarrow|\partial_{t} f|\leq c(n)K'w^{-2}f,
\end{equation}
on $[0,\hat{T}]$ and $s_{0}>M+2w$. Thence, combining with $f(x,0)\geq v$ whenever $s_{0}(x)\geq M+2w$, after shrink $\hat{T}$ again, $f(x,t)\geq\frac{v}{2}$ on $[0,\hat{T}]$, whenever $s_{0}(x)>M+2w$. Note that by the Cauchy-Schwartz inequality and (\ref{eq-of-f}),
\begin{equation}
\begin{aligned}
        &~\left(\frac{\partial}{\partial t}-\frac{\partial^{2}}{\partial s_{t}^{2}}\right)f^{n}=-n(n-1)f^{n-2}\geq-(n-1)\left[(n-2)f^{n}+2\right]\\
        \Rightarrow&~\left(\frac{\partial}{\partial t}-\frac{\partial^{2}}{\partial s_{t}^{2}}\right)\left[e^{(n-1)(n-2)t}\left(f^{n}+\frac{2}{n-2}\right)\right]\geq0.
\end{aligned}
\end{equation}
Restrict $f$ on $\{x\in\R^{n+1}: s_{0}(x)\in[1/2,M+2w]\}\times[0,\hat{T}]=:\Omega$. The above discussion shows that on the parabolic boundary of $\Omega$,
\begin{equation}
    e^{(n-1)(n-2)t}\left(f^{n}+\frac{2}{n-2}\right)\geq\min\left\{\frac{v}{2},\frac{\delta}{16},\varepsilon\right\}^{n}+\frac{2}{n-2},
\end{equation}
Then, after shrinking $\hat{T}$ again, the standard maximal principle shows there is a constant $\hat{\delta}(n,v,\delta,\varepsilon)>0$ such that $f\geq\hat{\delta}$ outside $B_{0}(o,1/2)\times[0,\hat{T}]$. By (\ref{eq-of-set-inclusion}), $f\geq\hat{\delta}$ holds outside $B_{t}(o,1)$ for all $t\leq\hat{T}$. Therefore, after setting $\tilde{\delta}:=\min\{\frac{\delta}{4},\hat{\delta}\}$, a constant only depends on $n,v,\delta,\varepsilon$, we get $f(s_{t},t)\geq\tilde{\delta}s_{t}\chi_{[0,1]}(s_{t})+\tilde{\delta}\chi_{(1,\infty)}(s_{t})$ on $\R^{n+1}\times[0,\hat{T}]$. Thus, we complete the proof by applying Proposition \ref{volumenoncollapseforrota} on $\R^{n+1}\times[0,\hat{T}]$.
\end{proof}

At the end of this section, we recall Shi's existence theorem of Ricci flows under the bounded curvature assumption.
\begin{lm}\cite[Theorem 1.1.]{10.4310/jdg/1214443292}\label{Shi'sexitencetheroem}
    There exists two dimension constants $\Lambda(n),a(n)>0$ such that, for any complete non-compact manifold $(\M,g)$ with curvature bound $|{\Rm}|\leq k_{0}$, there exists a complete Ricci flow $(\M,g(t))_{t\in[0,\frac{\Lambda}{k_{0}}]}$ starting from $g$ with
    \begin{equation}
        |{\Rm}(g(t))|_{g(t)}\leq ak_{0}\leq\frac{a\Lambda}{t},
    \end{equation}
    on $\M\times[0,\frac{\Lambda}{k_{0}}]$. 
\end{lm}
\section{Proof of Main results}\label{proof of ain theorem}
\begin{proof}[Proof of Theorem \ref{theoremofnoncollapsing}]
    After a parabolic rescaling, we may assume $\ell=1$. For any $k>10$, consider the rotationally symmetric metric $g_{k}=ds^{2}+f_{k}(s)^{2}g_{\text{std}}$ such that
    \begin{equation}
        \begin{cases}
            f_{k}(s)=f(s),\text{ for all }s\in[0,k];\\
            f_{k}\text{ is linear outside }[0,k+1];\\
            \frac{\partial}{\partial s}f_{k}(s)>0,\text{ on }\R^{n+1}.
        \end{cases}
    \end{equation}
    Then $(\R^{n+1},g_{k})$ is a complete and rotationally symmetric manifold with bounded curvature. By Lemma \ref{Shi'sexitencetheroem} and the uniqueness theorem of Ricci flow with bounded curvature \cite[Theorem 1.1.]{CZuniquenessofRF}, there is a complete and RSRF $(\R^{n+1},g_{k}(t))_{t\in[0,T_{k}]}$ starting from $g_{k}$ with bounded curvature and satisfying $|{\Rm}(g_{k}(t))|\leq c/t$, where $c=a\Lambda$ is the dimensional constant in Lemma \ref{Shi'sexitencetheroem}. Note that $T_{k}$ could tend to zero a priori, with no uniform curvature bound in spacetime. However, we claim they can be extended to a uniform time $\hat{T}$ and with an instantaneously uniform curvature bound.
    \begin{clm}
        There exist two constants $\hat{T}(n,\delta),\hat{\Lambda}(n,\delta)>0$ such that $(\R^{n+1},g_{k}(t))$ could be extended until $\hat{T}$ and satisfy
        \begin{equation}
            |{\Rm}(g_{k}(t))|\leq\frac{\hat{\Lambda}}{t},
        \end{equation}
        on $\R^{n+1}\times(0,\hat{T}]$.
    \end{clm}
    As long as the claim holds, combined with the equivariant compactness theorem of Ricci flow \cite[Proposition 3.7.]{BHZ24} and $g_{k}(0)$ converges to $g$ locally smoothly, there is a complete and RSRF $(\R^{n+1},g(t))_{t\in[0,\hat{T}]}$ starting from $g$ such that there is a sequence $k_{j}\rightarrow\infty$ so that $(\R^{n+1},g_{k_{j}}(t))_{t\in[0,\hat{T}]}$ smoothly converges to $(\R^{n+1},g(t))_{t\in[0,\hat{T}]}$. By virtue of the Claim, $|{\Rm(g(t))}|\leq\hat{\Lambda}/t$ holds on $\R^{n+1}\times(0,\hat{T}]$. It remains to confirm the claim. From now on, we omit the index $k$ for convenience and set $t_{1}:=T_{k}$ for simplicity. Since $(\R^{n+1},g(t))_{t\in[0,t_{1}]}$ has bounded curvature and $g(0)=g$ if $|s|\leq k$, by Lemma \ref{noncollapsingforboundedcurvature}, there are two constants $T_{1}(n,\delta),v(n,\delta)>0$ such that
    \begin{equation}\label{eq105}
        f_{s}\geq0\text{ and }\frac{\vol_{g(t)}B_{g(t)}(x,r)}{r^{n+1}}\geq v,
    \end{equation}
    on $\R^{n+1}\times[0,\min\{t_{1},\hat{T}\}]$ and for all $r\in(0,1]$. Note that $\text{W}(g(t))=0$ for all $t\in[0,t_{1}]$, by Lemma \ref{prioriestimate}, there are two constants $T_{2}(n,0,v),\hat{\Lambda}(n,0,v)>0$ such that
    \begin{equation}\label{eq106}
        |{\Rm}(g(t))|\leq\frac{\hat{\Lambda}}{t},
    \end{equation}
    on $\R^{n+1}\times(0,\min\{t_{1},T_{1},T_{2}\}]$. Choosing $\hat{T}:=\min\{T_{1},T_{2},T_{3}\}>0$ for some $T_{3}$ to be determined later. If $t_{1}\geq\hat{T}$, then we done. Otherwise, let $T_{\max}\geq t_{1}$ be the supremum of $T>0$ such that, there exits a complete and RSRF $(\R^{n+1},\tilde{g}(t))_{t\in[0,T]}$ such that 
    \begin{enumerate}
        \item\label{inductionhyp1} $\tilde{g}(t)$ has bounded curvature and satisfies (\ref{eq106}) on $\R^{n+1}\times(0,T]$.
        \item\label{inductionhyp2} $\tilde{g}(t)$ is an extension of $g(t)$.
    \end{enumerate}
    Now, at $t=T$, $(\R^{n+1},\tilde{g}(T))$ is a complete and rotationally symmetric manifold with curvature bound $\hat{\Lambda}/T$, by Lemma \ref{Shi'sexitencetheroem} and uniqueness theorem again, there is a complete and RSRF $(\R^{n+1},h(t))_{t\in[0,\frac{\Lambda}{\hat{\Lambda}}T]}$ starting from $\tilde{g}(T)$ with curvature bound $a\frac{\hat{\Lambda}}{T}$, where $a$ is the constant in Lemma \ref{Shi'sexitencetheroem}. Extend $\tilde{g}$ by $h$, we construct an extension $\tilde{g}'(t)$ on $[0,T':=(1+\frac{\Lambda}{\hat{\Lambda}})T]$ with curvature bound
    \begin{equation}
        |{\Rm}(\tilde{g}'(t))|\leq\frac{\max\{\hat{\Lambda},a(\hat{\Lambda}+{\Lambda})\}}{t},
    \end{equation}
    on $\R^{n+1}\times(0,T']$. By Lemma \ref{noncollapsingforboundedcurvature}, there is a constant $T_{3}(n,\max\{\hat{\Lambda},a(\hat{\Lambda}+{\Lambda})\},\delta)=T_{3}(n,\delta)>0$ such that (\ref{eq105}) holds for $\tilde{g}'(t)$ on $\R^{n+1}\times[0,\min\{T',T_{3}\}]$. By Lemma \ref{prioriestimate}, (\ref{eq106}) holds for $\tilde{g}'(t)$ on $\R^{n+1}\times(0,\min\{T',\hat{T}\}]$. Choose $T\in((1+\frac{\Lambda}{\hat{\Lambda}})^{-1}T_{\max},T_{\max})$, and this implies $T_{\max}\geq\hat{T}$ by the choice of $T_{\max}$. Therefore, we construct an extension satisfying the claim, and this completes the proof.
\end{proof}
\begin{rmk}
    A priori, we only have $f_{s}\geq0$ on $\R^{n+1}\times[0,\hat{T}]$. Nevertheless, if $f_{s}=0$ at some point $(p,t_{0})$ with $t_{0}>0$, then the strong maximal principle would imply $f_{s}\equiv0$ on $\R^{n+1}\times[0,t_{0}]$, which is impossible. Therefore, the warped function $f$ would instantaneously satisfy $f_{s}>0$.
\end{rmk}
\begin{proof}[Proof of Corollary \ref{long-time solution}]
   By Theorem \ref{theoremofnoncollapsing}, for any $\ell>0$, there is complete and RSRF $(\R^{n+1},g_{\ell}(t))_{t\in[0,T_{g}\ell^{2}]}$ with $g_{\ell}(0)=g$ such that (\ref{eq-of-cl1}) holds for $g_{\ell}$ and on $\R^{n+1}\times(0,T_{g}\ell^{2}]$. Then the solution comes from taking a subsequential limit of $g_{\ell}$. An alternative viewpoint is to apply the uniqueness theorem \cite[Theorem 1.1.]{lee2025uniquenessricciflowscaling}, which implies $g_{\ell_{1}}\equiv g_{\ell_{2}}$ on $\R^{n+1}\times[0,T_{g}\ell_{1}^{2}]$ for all $\ell_{1}\leq\ell_{2}$.  
\end{proof}

The proof of Theorem \ref{thm-of-nondegenerated-warped-function} is almost the same as the proof of Theorem \ref{theoremofnoncollapsing}, but with a different choice of approximation.
    
\begin{proof}[Proof of Theorem \ref{thm-of-nondegenerated-warped-function}]
After a parabolic rescaling, we may assume $\ell=1$. For $k>200$ large enough, depending on $\varepsilon$, we consider a function $f_{k}:[0,\infty)\rightarrow[0,\infty)$ such that
\begin{equation}
    \begin{cases}
        f_{k}=f,\text{ on }[0,k];\\
        f_{k}\geq\frac{\varepsilon}{2},\text{ on }[k,k+1];\\
        f_{k}=\frac{\varepsilon}{2},\text{ on }[k+1,\infty)
    \end{cases}
\end{equation}
    and define the corresponding metric $g_{k}:=ds^{2}+f_{k}(s)^{2}g_{\text{std}}$. Since $(\R^{n+1},g_{k})$ is complete and having bounded curvature, by Lemma \ref{Shi'sexitencetheroem}, there is a complete and RSRF $(\R^{n+1},g_{k}(t))_{t\in[0, T_{k}]}$ with $g_{k}(0)=g_{k}$ and $|{\Rm(g_{k}(t))}|\leq\frac{a\Lambda}{t}$ on $\R^{n+1}\times(0,T_{k}]$. Similar to the proof of Theorem \ref{theoremofnoncollapsing}, we aim to extend the flow $g_{k}(t)$ until a uniform time $[0,T_{g}]$ via the pseudolocality Lemma \ref{prioriestimate} and the a priori estimate Lemma \ref{noncollapsing for bounded curvature v2} with the scaling invariant curvature bound $|{\Rm}(g_{k}(t))|\leq\Lambda_{g}/t$. We provide the proof here for completeness, although it is the same as before.

    By Lemma \ref{noncollapsing for bounded curvature v2} with $\ell=1$, $M\gtrsim k+4w$, $v=\frac{\varepsilon}{2}$, and $w=w(\varepsilon)>0$ (the curvature of cylinder with width $\varepsilon$), there are two constant $v_{1}(n,\varepsilon,\delta)>0$ and $\tilde{T}_{1}(n,\varepsilon,a\Lambda,\delta)>0$ such that 
    \begin{equation}\label{eq-of-volume ratio of v1}
        \frac{\vol_{g_{k}(t)}(B_{g_{k}(t)}(x,r))}{r^{n+1}}\geq v_{1},
    \end{equation}
    on $\R^{n+1}\times[0,\min\{\tilde{T}_{1},T_{k}\}]$. From now on, we omit the index $k$. By Lemma \ref{prioriestimate}, there are two constants $\tilde{T}_{2}(n,v_{1}),\hat{\Lambda}(n,v_{1})>0$ such that
    \begin{equation}\label{eq-of-thmv2}
        |{\Rm(g(t))}|\leq\frac{\hat{\Lambda}}{t},
    \end{equation}
    on $\R^{n+1}\times[0,\min\{T,\tilde{T}_{1},\tilde{T}_{2}\}]$. Choosing $\tilde{T}:=\min\{\tilde{T}_{1},\tilde{T}_{2},\tilde{T}_{3}\}>0$ for some $\tilde{T}_{3}$ to be determined later, and setting $t_{1}:=T$. If $t_{1}\geq\tilde{T}$, then we done. Otherwise, let $T_{\max}\geq t_{1}$ be the supremum of $T>0$ such that, there exits a complete and RSRF $(\R^{n+1},\tilde{g}(t))_{t\in[0,T]}$ such that 
    \begin{enumerate}
    \item\label{inductionhyp1v2} $\tilde{g}(t)$ has bounded curvature and satisfies (\ref{eq-of-thmv2}) on $\R^{n+1}\times(0,T]$.
        \item\label{inductionhyp2v2} $\tilde{g}(t)$ is an extension of $g(t)$.
    \end{enumerate}
    Now, at $t=T$, $(\R^{n+1},\tilde{g}(T))$ is a complete and rotationally symmetric manifold with curvature bound $\hat{\Lambda}/T$, by Lemma \ref{Shi'sexitencetheroem} and uniqueness theorem again, there is a complete and RSRF $(\R^{n+1},h(t))_{t\in[0,\frac{\Lambda}{\hat{\Lambda}}T]}$ starting from $\tilde{g}(T)$ with curvature bound $a\frac{\hat{\Lambda}}{T}$, where $a$ is the constant in Lemma \ref{Shi'sexitencetheroem}. Extend $\tilde{g}$ by $h$, we construct an extension $\tilde{g}'(t)$ on $[0,T':=(1+\frac{\Lambda}{\hat{\Lambda}})T]$ with curvature bound
    \begin{equation}
        |{\Rm}(\tilde{g}'(t))|\leq\frac{\max\{\hat{\Lambda},a(\hat{\Lambda}+{\Lambda})\}}{t},
    \end{equation}
    on $\R^{n+1}\times(0,T']$. By Lemma \ref{noncollapsing for bounded curvature v2}, there is a constant $\tilde{T}_{3}(n,\max\{\hat{\Lambda},a(\hat{\Lambda}+{\Lambda})\},\varepsilon,\delta)=\tilde{T}_{3}(n,\varepsilon,\delta)>0$ such that (\ref{eq-of-volume ratio of v1}) holds for $\tilde{g}'(t)$ on $\R^{n+1}\times[0,\min\{T',\tilde{T}_{3}\}]$. By Lemma \ref{prioriestimate}, (\ref{eq-of-thmv2}) holds for $\tilde{g}'(t)$ on $\R^{n+1}\times(0,\min\{T',\tilde{T}\}]$. Choose $T\in((1+\frac{\Lambda}{\hat{\Lambda}})^{-1}T_{\max},T_{\max})$, and this implies $T_{\max}\geq\tilde{T}$ by the choice of $T_{\max}$. Therefore, we could extend the flow until $\tilde{T}$ with a uniform estimate (\ref{eq-of-thmv2}).

    Since $g_{k}$ converges to $g$ smoothly locally, by the equivariant compactness theorem \cite[Proposition 3.7.]{BHZ24} and \cite[Corollary 3.2.]{10.4310/jdg/1246888488}, we acquire a subsequential limit of $(\R^{n+1},g_{k}(t),o)_{t\in[0,\tilde{T}]}$, a complete and RSRF $(\R^{n+1},g(t))_{t\in[0,\tilde{T}]}$ with $g(0)=g$ and
    \begin{equation}
        |{\Rm}(g(t))|\leq\frac{\hat{\Lambda}}{t},
    \end{equation}
    on $\R^{n+1}\times(0,\tilde{T}]$. Set $\Lambda_{g}:=\hat{\Lambda}$ and $T_{g}:=\tilde{T}$, we complete the proof.
\end{proof}

\section{Proof of Corollary \ref{roughinitialdata}}\label{approxmateion section}

We first introduce the evolution equation of the Ricci lower bound along RSRF.
\begin{lm}\label{eq of Ricci lower bound}
    Let $(\M,g(t))_{t\in[0,T]}$ be a Ricci flow with vanishing Weyl tensor on $\M\times[0,T]$. Define a continuous function
    \begin{equation}
        \ell(x,t):=\inf\{s>0:\ \Ric(x,t)+sg(x,t)\geq0\},
    \end{equation}
    for all $(x,t)\in\M\times[0,T]$. Then
    \begin{equation}
        \left(\frac{\partial}{\partial t}-\Delta_{g(t)}\right)\ell\leq\scal\cdot\ell+c(n)\ell^{2},
    \end{equation}
    in the sense of barrier, where $c(n)>0$ is some universal constant.
\end{lm}
\begin{proof}[Proof of Lemma \ref{eq of Ricci lower bound}]
    Recall that from the computation in \cite[Section 4]{ZHANG2015774}, if we diagonalize the Ricci curvature with eigenvalues $\lambda_{1}\leq\cdots\leq\lambda_{n}$, then the evolution equation of Ricci tensor becomes
    \begin{equation}
        \left(\frac{\partial}{\partial t}-\Delta_{g(t)}\right)\Ric_{ii}=\frac{2}{n-2}\sum_{k\neq i}\left(\lambda_{k}+\lambda_{i}-\frac{\scal}{n-1}\right)\lambda_{k}.
    \end{equation}
    We may assume $\ell(x_{0},t_{0})>0$ at $(x_{0},t_{0})$, otherwise we choose $0$ be the barrier function near $(x_{0},t_{0})$. Take $E\in T_{x_{0}}\M$ being the eigenvector of $\Ric(x_{0},t_{0})$ with eigenvalue $\lambda_{1}$. We adopt the Uhlenbeck's trick to extend $E$ parallel near $x_{0}$ at $t=t_{0}$ and satisfying $\frac{\partial}{\partial t}E=\Ric(E)$, then $|E|=1$ near $(x_{0},t_{0})$ and $\nabla E=\Delta E=0$ at $(x_{0},t_{0})$. Define the function  $\phi(x,t):=-\Ric(E,E)\geq \ell(x,t)$ and equality holds at $(x_{0},t_{0})$. At $(x_{0},t_{0})$,
    \begin{equation}
        \begin{aligned}
            \left(\frac{\partial}{\partial t}-\Delta_{g(t_{0})}\right)\phi(x_{0},t_{0})=&~-\frac{2}{n-2}\sum_{k=2}^{n}\left(\lambda_{k}+\lambda_{1}-\frac{\scal}{n-1}\right)\lambda_{k}-2\lambda_{1}^{2}.
        \end{aligned}
    \end{equation}
    Now, according to \cite[Corollary 1.2]{CM09}, either $\lambda_{1}=\cdots=\lambda_{n-1}$ or $\lambda_{2}=\cdots=\lambda_{n}$ holds. In the first case, the R.H.S. becomes
    \begin{equation}
    \begin{aligned}
        R.H.S.=&~ -2\lambda_{1}(\lambda_{1}-\frac{\lambda_{n}}{n-1})-\frac{2}{n-1}\lambda_{n}^{2}-2\lambda_{1}^{2}\\
        =&~-\frac{2}{n-1}\scal^{2}-\frac{4n-2}{n-1}\scal\cdot\phi-(2n+4)\phi^{2}\\
        \leq&~\scal\cdot\phi+\frac{(5n-3)^{2}-16(n+2)(n-1)}{8(n-1)}\phi^{2},
    \end{aligned}
    \end{equation}
    where we apply the Cauchy-Schwartz inequality in the last inequality. In the second case, 
    \begin{equation}
    \begin{aligned}
        R.H.S. =&~-2\lambda_{1}\lambda_{n}-2\lambda_{1}^{2}\\
        =&~(-\lambda_{1})\scal+(n-3)\lambda_{1}(\lambda_{n}-\lambda_{1})+(n-4)\lambda_{1}^{2}\\
        \leq&~\scal\cdot\phi+(n-4)\phi^{2},
    \end{aligned}
    \end{equation}
    where the last line follows from $\lambda_{1}<0$ and $\lambda_{1}\leq\lambda_{n}$. Take \\ $c(n):=\max\left\{\frac{(5n-3)^{2}-16(n+2)(n-1)}{8(n-1)},n-4\right\}$ a dimensional constant, and this completes the proof. 
\end{proof}
After applying the local maximal principle due to Lee-Tam \cite[Theorem 1.2.]{Lee_Tam_2022}, we are able to control the Ricci lower bounds along RSRF.
\begin{lm}[A priori estimate of Ricci lower bounds]\label{preservation-Ricci lower bound}
Let $(\R^{n+1},g(t))_{t\in[0,T]}$ be a complete and RSRF with $\Ric(g(0))\geq-L$ on $B_{g(0)}(p,r)$. Suppose that there is a constant $\Lambda>0$ so that $|{\Rm}(g(t))|\leq\Lambda/t$ on $\R^{n+1}\times(0,T]$. Then there is a constant $L'(n,L,\Lambda,T)>0$ such that $\Ric(p,t)\geq-L'(1+r^{-2})$ on $\forall t\in[0,T]$. 
\end{lm}
\begin{proof}[Proof of Lemma \ref{preservation-Ricci lower bound}]
    By \cite[Proposition 2.4.]{BHZ24}, for all $t\in[0,T]$, $(\R^{n+1},g(t))$ is $\kappa(n)$-noncollapsed on all scale. Combining this with the curvature bound $\Lambda/t$, by the Cheeger-Gromov-Taylor's classical result, there is a constant $c'(n)>0$ such that $\inj_{g(t)}(\R^{n+1})\geq c'\sqrt{t/\Lambda}$ for all $t\in[0,T]$. Then the assertion follows from Lemma \ref{eq of Ricci lower bound} and \cite[Theorem 1.2.]{Lee_Tam_2022}.
\end{proof}

The basic idea is to find an approximation $g_{k}:=ds^{2}+f_{k}(s)^{2}g_{\text{std}}$ so that $f_{k}\rightarrow f$ smoothly locally on $\R^{n+1}\setminus\{o\}$ as $k\rightarrow\infty$, and $\scal(g_{k})$ have a uniform lower bound near the origin. This is predictable since $-2ff_{ss}+(n-1)(1-f_{s}^{2})\sim(n-1)(1-f_{s}(0)^{2})>0$ near $s=0$, it should be possible to smooth out $g$ preserving the lower bound of scalar curvature. In practice, the scalar curvature is positive near the vertex. On the other hand, the lower bounds of Ricci curvature ensure the limiting flow converges to the initial data \textit{smoothly} outside the origin.

\begin{proof}[Proof of Corollary \ref{roughinitialdata}]
    Since $f(0)=f^{(2)}(0)=0$ and $v:=f_{s}(0)\in(0,1)$, there is a $\delta>0$ such that $|f(s)|^{2},|f_{ss}(s)|^{2}\leq \frac{1-\frac{v^{2}}{4}}{100}$ for all $s\in[0,\delta]$. Let $\psi:[0,1]\rightarrow[0,1]$ be a smooth function such that
    \begin{equation}
        \begin{cases}
            \psi^{(k)}(1)=0,\text{ for all }k\geq0;\\
            \psi^{(k)}(0)=0,\text{ for all }k\geq1;\\
            \psi(0)=1;\\
            \psi'(s)\leq0,\text{ on }[0,1];\\
            |\psi'|+|\psi''|+|\psi'''|\leq C,\text{ for some constant }C>0. 
        \end{cases}
    \end{equation}
    We may assume $f_{s}\in[\frac{v}{2},\frac{v+1}{2}]$ on $[0,\delta]$. Then $-2ff_{ss}+(n-1)(1-f_{s}^{2})\geq(n-2)(1-\frac{v^{2}}{4})>0$ on $s\in[0,\delta]$. For any $k>\delta^{-1}$, we choose 
    \begin{equation}
        \varepsilon_{k}:=\min\left\{\frac{1}{Ck^{2}}\sqrt{\frac{1-\frac{v^{2}}{16}}{100}},\frac{v}{4Ck},\frac{vL}{2k^{3}C}\right\},
    \end{equation}
    and consider the function
    \begin{equation}
        F_{k}(s):=\varepsilon_{k} \psi(ks)+f(s).
    \end{equation}
    Then $F_{k}^{(1)}(s)\geq\frac{v}{4}$ on $[0,\infty)$ and on $[0,\frac{1}{k}]$,
    \begin{equation}
        -2F_{k}F_{k}^{(2)}+(n-1)\left(1-(F_{k}^{(1)})^{2}\right)\geq -2\cdot2\frac{1-\frac{v^{2}}{16}}{100}+(n-1)\left(1-\frac{v^{2}}{16}\right)\geq(n-2)\left(1-\frac{v^{2}}{16}\right).
    \end{equation}
    Note that $F_{k}(0)=\varepsilon_{k}$ and $F_{k}^{(\ell)}(0)=f^{(\ell)}(0)$ for all $\ell\geq1$, this allows us to smoothing the origin. On the other hand, if $F_{k}^{(2)}(s_{0})\geq0$, then at $s_{0}$,
    \be
    -\frac{F_{k}^{(2)}}{F_{k}}\geq-\frac{k^{2}\varepsilon_{k}\phi^{(2)}+f^{(2)}}{f}\geq-\frac{k^{3}\varepsilon_{k}Cs_{0}}{\frac{v}{2}s_{0}}-\frac{L}{n}\geq-2L.
    \ee
    This implies $\Ric(ds^{2}+F_{k}^{2}(s)g_{\text{std}})\geq-2L$ on $(\R^{n+1}\setminus\{o\})\times[0,T]$.
\medskip

    Take a smooth function $\Psi:[0,1]\rightarrow\R_{\geq0}$ with 
    \begin{equation}
        \begin{cases}
            \Psi^{(\ell)}(0)=0,\text{ for all }\ell\neq1;\\
            \Psi^{(1)}(0)=1;\\
            \Psi^{(\ell)}(1)=0,\text{ for all }\ell\geq2;\\
            \Psi^{(1)}(1)=v;\\
            \Psi^{(2)}(s)\leq0,\text{ for all }s\in[0,1].
        \end{cases}
    \end{equation}
    (For example, let $-\Psi^{''}$ be a cut-off function whose integral is $1-v$ and supported in $[0,1]$.) Thus we define a smooth warped function $f_{k}:[0,\infty)\rightarrow[0,\infty)$ such that 
    \begin{equation}
        f_{k}(s):=\begin{cases}
            \frac{\varepsilon_{k}}{\Psi(1)}\Psi\left(\frac{\Psi(1)}{\varepsilon_{k}}s\right),\text{ if }s\in[0,\frac{\varepsilon_{k}}{\Psi(1)}];\\
            F_{k}(s-\frac{\varepsilon_{k}}{\Psi(1)}),\text{ otherwise,}
        \end{cases}
    \end{equation}
    where smoothness follows from the initial condition on $f$ at $s=0$. By our construction, $g_{k}:=ds^{2}+f_{k}(s)^{2}g_{\text{std}}$ is a smooth metric on $\R^{n+1}$ and have uniform lower bound (independent of $k$) on $\scal(g_{k})$ and $f_{k}$ on $B_{g}(o,4)$ (Note that the Ricci curvature is greater than $-2L$ on $B_{g}(o,\varepsilon_{k}+1/k)$). By Theorem \ref{theoremofnoncollapsing}, there are two constants $\Lambda(n,g),T(n,g)$ and a complete and RSRF $(\R^{n+1},g_{k}(t))_{t\in[0,T]}$ with $g_{k}(0)=g_{k}$ and satisfying
    \begin{equation}\label{eq46}
        |{\Rm(g_{k}(t))}|\leq\frac{\Lambda}{t},
    \end{equation}
    on $\R^{n+1}\times(0,T]$. Also, since $\Ric(g_{k}(0))\geq-2L$ holds on $\R^{n+1}$, there is a constant $L'(L,\Lambda,n)>0$, independent of $k$, such that $\Ric(g_{k})\geq-L'$ on $\R^{n+1}\times[0,T]$ for all $k$. 
\medskip
    
    Now, by the equivariant compactness theorem \cite[Proposition 3.7.]{BHZ24}, after passing to subsequence, there is a pointed, complete, RSRF $(\M_{\infty},g_{\infty}(t),p_{\infty})_{t\in(0,T]}$ such that
    \be
    (\R^{n+1},g_{k}(t),o)_{t\in(0,T]}\overset{\text{eC-G convergence}}{\longrightarrow}(\M_{\infty},g_{\infty}(t),p_{\infty})_{t\in(0,T]},
    \ee
    as $k\rightarrow\infty$. Since $p_{\infty}$ lies on the singular orbit and the manifold $\M_{\infty}$ is unbounded, $(\M_{\infty},p_{\infty})=(\R^{n+1},o)$ follows from \cite[Proposition 2.2.]{BHZ24}. Now, by $\Ric(g_{k})\geq-L'$ and $|{\Rm(g_{k}(t)}|\leq\Lambda/t$, the distance distortion Lemma \ref{distancelemma} shows that
    \be
    d_{g_{k}(0)}(x,y)-\beta\sqrt{\Lambda t}\leq d_{g_{k}(t)}(x,y)\leq e^{2L't}d_{g_{k}(0)}(x,y),
    \ee
    for all $x,y\in\R^{n+1}$. Also, since $(\R^{n+1},d_{g_{k}(0)})$ converges to $(\R^{n+1},d_{g})$, this confirms the $\mathcal{C}^{0}$-convergence from $(\R^{n+1},d_{g_{\infty}(t)})$ to $(\R^{n+1},d_{g})$, the metric completion of $(\R^{n+1}\setminus\{o\},g)$, as $t\rightarrow0$. By the smooth convergence, we still have $\Ric(g_{\infty}(t))\geq-L'$ on $\R^{n+1}\times(0,T]$. This enhances the regularity of convergence at $t=0$ outside the vertex by \cite[Theorem 1.6.]{DSS22}, that is, $(\R^{n+1}\setminus\{o\},g_{\infty}(t))$ smoothly converges to $(\R^{n+1}\setminus\{o\},g)$ as $t\rightarrow0$. This confirms the first assertion. The second case follows from our construction, since $\scal(g_{k})\geq0$ and $f^{(1)}_{k}\geq v/4$ hold on $\R^{n+1}$ if $\scal(g)\geq0$ and $f_{s}\geq v$ on $\R^{n+1}$. This completes the proof.
\end{proof}
\begin{rmk}
    It might be possible to weaken the boundary condition on $f$. We require the cone-like condition just for the convenience of gluing construction. The conditions $f(0)=f_{ss}(0)=0$ and $f_{s}(0)\in(0,1)$ should be enough.
\end{rmk}

\end{document}